\documentclass[12pt]{amsart}

\usepackage{amsfonts,
latexsym,
amssymb,
}

\newcommand{\labbel}{\label}

\newtheorem{theorem}{Theorem}[section]
\newtheorem{lemma}[theorem]{Lemma}
\newtheorem{thm}[theorem]{Theorem}
\newtheorem{prop}[theorem]{Proposition} 
\newtheorem{proposition}[theorem]{Proposition} 
\newtheorem{cor}[theorem]{Corollary} 
\newtheorem{corollary}[theorem]{Corollary}

\theoremstyle{definition}
\newtheorem{definition}[theorem]{Definition}

\newtheorem{problems}[theorem]{Problems} 

\theoremstyle{remark}
\newtheorem{remark}[theorem]{Remark}

\DeclareMathOperator{\Thetax}{\Theta}
\DeclareMathOperator{\Psix}{\Psi}

\DeclareMathOperator{\betax}{\beta}
\DeclareMathOperator{\Mx}{M}
\DeclareMathOperator{\Rx}{R}

\newcommand{\m}[1]{ \mathfrak {\uppercase {#1}}}

 \newcommand{\conv}{{^{-}}}

\begin{document}
 
\title
{Representable tolerances in varieties}

\author{Paolo Lipparini} 
\address{Dipartimento Novello di Matematica\\Viale della Ricerca  Scient\` \i fica\\II Universit\`a di Roma (Tor Vergata)\\I-00133 ROME ITALY}
\urladdr{http://www.mat.uniroma2.it/\textasciitilde lipparin}

\thanks{We wish to express our gratitude to G. Cz{\'e}dli and E. W. Kiss 
for stimulating discussions and correspondence.}  



\begin{abstract}
We discuss two possible ways of representing tolerances: 
first, as a homomorphic image of some congruence;
second,
as the relational composition of  some compatible relation
with its converse.
 The second way is independent from the variety under consideration, while
the first way is variety-dependent. The relationships between
these two kinds of representations are clarified. 

As an application, we show that any tolerance on some
lattice $\m L $ is the image of some congruence on a subalgebra of 
$\m L \times \m L$. This is related to
recent results by G. Cz{\'e}dli and E. W. Kiss \cite{CK}. 
\end{abstract}

\maketitle

\section{Introduction} \labbel{intro}

Recall that a \emph{tolerance} on some algebra is a binary, compatible, symmetric and reflexive relations. Thus, a congruence is just a transitive tolerance.

It is quite surprising that the study of tolerances (apart from intrinsic interest) has revealed to be 
  essential in the study of congruences. Indeed, present-day research shows that tolerances
are becoming increasingly important even in many other
at first look seemingly unrelated contexts
\cite{Ch,CCH,CCHL,CG,CHL,CK,Gu,HM,HMK,
KK,acta,
Sm,We}.

\section{A first example} \labbel{first}

\begin{definition} \labbel{reprhom}
Suppose that  $\m a$, $\m b$ are algebras, 
and 
$\varphi : \m b \to \m a$
 is a surjective homomorphism.
It is folklore, and easy to see, that if 
$\beta$ is a congruence on $\m b$, 
then $\varphi ( \beta ) = \{ ( \varphi (a), \varphi(b)) \mid a \betax b \in B\} $ 
is a tolerance $\Theta$ on  $\m a$.
In the above situation, we say that \emph{$\Theta$ is an image of $\beta$}. In case we need to specify $\m b$ explicitly,
we shall say that 
\emph{$\Theta$ is the image of a congruence on $\m b$}.
 \end{definition}   

We now exemplify our methods in the particular case of lattices.

\begin{theorem} \labbel{lattices}
If $\m  L$ is a lattice and $\Theta$ is a tolerance on $\m L$,
then $\Theta$ is the image of some congruence on some subalgebra of 
$\m L \times \m L$.      
 \end{theorem}

 \begin{proof} 
The partial order  $\leq$ induced by the lattice operations is a compatible  relation on $\mathcal L$, thus 
also $\leq \cap~  \Thetax $  is compatible.
Hence the binary relation $\leq \cap~ \Thetax $ can be considered as a subalgebra $\m b$ of $\m L \times \m L$.
 Let $\varphi : \m b \to \m a$ be the first projection, and $\beta$ on $\m b$ be the 
kernel of the second projection.

We shall show that $\varphi ( \beta ) =\Theta$. Indeed, if
$a \Thetax b$, then $a = a \vee a \Thetax a \vee b$,  
thus $(a,  a \vee b) \in B$, since  $a \leq a \vee b$. Similarly,  $(b,  a \vee b) \in B$. 
Trivially,  $(a,  a \vee b) \betax~ (b,  a \vee b) $, 
$ \varphi (a,  a \vee b) =a$, $  \varphi (b,  a \vee b)=b $, thus 
$ \Theta \subseteq \varphi( \beta )$.

Conversely, if $(a, b) \in \varphi( \beta )$,
then there is $ c \in L$ such that $(a, c) \in B$, $(b, c) \in B$,
hence $a \Thetax c$, $a \leq c$,      
$c \Thetax b$, $b \leq c$, 
thus
$a = a \wedge c \Thetax c \wedge b = b$. 
Hence $  \varphi( \beta ) \subseteq \theta $, and the theorem is proved.
\end{proof}  

Notice that we have not used all the properties of a lattice, thus Theorem \ref{lattices} allows some strengthening, see  Proposition \ref{lattgen} below.
Also, the proof of Theorem \ref{lattices} applies not only to lattices, but also to 
lattices with additional operations, provided the additional operations respect the lattice order, that is, the order remains a compatible  relation with respect to the 
additional operations. 

Notice also that if $\m L$ belongs to some variety of lattices $\mathcal V$, then 
every subalgebra of $\m L \times \m L$ belongs to $\mathcal V$. 
In fact, we do not need the full assumption that $\mathcal V$ is a variety:
we get that if  $\mathcal V$ is a class of lattices, 
and $\mathcal V$ is closed
under subalgebras and finite products
then every tolerances on some lattice in $\mathcal V$ is the image of some congruence of some lattice in $\mathcal V$.

Following \cite{CCH,CK}, we say that \emph{the tolerances of $\mathcal V$ are
images of its congruences} if every
tolerance in $ \mathcal V$
is the image of a congruence in $\mathcal V$.
By G. Cz{\'e}dli and G. Gr{\"a}tzer \cite{CG} the variety of all lattices, and by G. Cz{\'e}dli and
 E. W. Kiss \cite{CK}, every variety of lattices have the above property.
Theorem \ref{lattices} thus furnishes  another proof of   Cz{\'e}dli and
 E. W. Kiss result.

As we shall see in the next section, the reason why Theorem \ref{lattices} holds  is that 
tolerances in lattices are representable in the  sense of  Definition \ref{repr}, to be  introduced in the next section.

\section{Representable tolerances} \labbel{reprsec} 

\begin{definition} \labbel{repr}    
Suppose that  $\m a$ is an algebra, and $R$ is a compatible
reflexive relation on $\m a$. Let $R \conv$  denote the \emph{converse} of $R$.
It is immediate to see that $R \circ R \conv$ is a tolerance on 
$\m a$. Tolerances which are representable in the form  
$R \circ R \conv$ as above have been called \emph{representable}
in \cite[Definition 2]{acta}, where it has been shown, among other things, that not
every tolerance is representable. See  \cite[Section 6]{acta} 
 for more
on representable tolerances.
\end{definition}

\begin{theorem} \labbel{thm} 
If $\Theta$ is  a representable tolerance on the algebra $\m a$, then
$\Theta$ 
is the image of a congruence on some subalgebra of
$\m a \times \m a$.
\end{theorem} 

\begin{proof}
Suppose that 
$\Theta$  is
representable
as $\Theta= R \circ R\conv$, 
for $R$ a reflexive compatible relation on $\m A$.
Let $\m b$  be $R$ itself, considered as a subalgebra of 
$\m a \times \m a$, and let $\varphi$ be the first projection.
Since $R$ is reflexive, we have that $\varphi$ is surjective.
Let $\beta$ on $\m b$ be the kernel of the second projection,
that is, 
$(a,b) \beta (c,d)$ if and only if $b=d$.  We shall show that 
$\varphi( \beta )= \Theta$.

Indeed, for every $a,c \in A$, and 
 since  $\Theta= R \circ R\conv$,
the following is a chain of equivalent conditions.
  \begin{enumerate} 
   \item   
 $a \Thetax c$;
\item
There is $b\in A$ such that 
$(a , b) \in R$ and $ (b, c) \in R\conv$;
\item
There is $b \in A$ such that 
$(a , b) , (c,b) \in R$.
\item
There is $b\in A$ such that 
 $(a,b), (c,b) \in B$
(thus,  $(a,b) \beta  (c,b) $).
\item
 $(a, c) \in \varphi( \beta )$.
\end{enumerate} 
 
We have shown that $\Theta=\varphi( \beta )$, thus the theorem is proved.
\end{proof}  

G. Cz{\'e}dli observed that every 
 tolerance on a lattice is representable, as a consequence of 
Lemma 2 in \cite{Clattices}. Cf. also \cite{CZ}.
See \cite[Proposition 11]{acta} and 
Proposition \ref{lattgen} below, for some slightly more general results. 
Hence Theorem \ref{lattices} is actually a particular case of
Theorem \ref{thm}.  We have given a direct proof of Theorem \ref{lattices} since  it is relatively short and simple.  

It has been shown in \cite{CCHL,CK} that every tolerance 
on some algebra $\m a$ is the image of some congruence $\beta$, for  appropriate $\m b$ and $\varphi$.
However, in most cases,
$\m a$ belongs to some specified variety $\mathcal V$,
and it is a natural request to ask that $\m b$, too, belongs to    
$\mathcal V$. This observation justifies the next definition.

A note on terminology: we shall say that a tolerance $\Theta$  is in a variety $\mathcal V$ to mean that $\Theta$ is a tolerance
 on some algebra $\m a \in \mathcal V$. Technically,
this is justified since a tolerance  on $\m a$  can be seen as a subalgebra 
of $\m a \times \m a$  (and $\Theta$ and $\m a$ generate the same variety,
since, in the above sense, $\m a$ is isomorphic to a substructure of $\Theta$). A similar remark applies to congruences
in place of tolerance.

\begin{definition} \labbel{repvar}
If $\Theta$ is a tolerance on $\m a \in \mathcal V$,
we say that 
\emph{$\Theta$ is the image of a congruence in $\mathcal V$}
if it is possible to chose  $\m b \in \mathcal V$,
$\beta$  a congruence on $\m b$, 
and 
$\varphi : \m b \to \m a$
 a surjective homomorphism 
such that $\Theta = \varphi ( \beta )$.

The above definition is a local version
of the mentioned notion from  \cite{CCH,CK} that \emph{the tolerances of $\mathcal V$ are
 images of its congruences}.

Though in the above definitions $\mathcal V$ is intended to stand for a 
variety,  our results generally hold for an arbitrary class $\mathcal V$ 
which is closed under taking subalgebras and products, in particular,
for quasivarieties. Actually, in most cases, it is enough to assume that $\mathcal V$ is closed under taking subalgebras and finite products. 
 \end{definition}   

As an immediate consequence of Theorem \ref{thm}, we get: 

\begin{corollary} \labbel{thmim} 
If all tolerances in the variety $\mathcal V$ are representable, then 
the tolerances of $\mathcal V$ are the
 images of its congruences
(actually, it is enough to suppose that 
$\mathcal V$ 
is closed under taking subalgebras and finite products).
\end{corollary}  

The converse of Corollary \ref{thmim} is not true. By \cite{CCHL,CK},
in
every variety defined by the empty set of equations
the tolerances of $\mathcal V$ are the
 images of its congruences, but, by
\cite[Proposition 10]{acta} (see also Proposition \ref{ex} below), there exists a non representable tolerance
on some algebra
(which trivially belongs to a variety defined by an empty set of equations).

However, it is possible to show that, within a given variety,  a tolerance is the image of
some congruence if and only if it is the image of some representable tolerance
(see Corollary \ref{imtol} below).
This can be useful, since if we want to show that, for a variety $\mathcal V$,
the tolerances of $\mathcal V$ are the
 images of its congruences, it is enough to show 
that the tolerances of $\mathcal V$ are the
 images of its representable tolerances.

\section{Weakly representable tolerances} \labbel{weaksec} 

There is a version of Theorem \ref{thm} dealing with a more general notion
of representability.

\begin{definition} \labbel{weakly}   
Let $\lambda$ be a nonzero cardinal.

We say that a tolerance is \emph{$\lambda$-weakly representable} if it
is the intersection of at most $\lambda$  representable tolerances. 

A tolerance is \emph{weakly representable} if it is
$\lambda$-weakly representable, for some nonzero cardinal $\lambda$.
Thus, representable is the same as $1$-weakly representable.
See again \cite[Section 6]{acta} 
 for more
informations about weakly representable tolerances.
\end{definition} 

In the statement of the next theorem, $+$ denotes \emph{cardinal sum},
that is, $\lambda+1= \lambda $, if $\lambda$ is infinite.

\begin{thm} \labbel{propweakly} 
If $\Theta$ is  a $\lambda$-weakly representable tolerance on the algebra $\m a$, then
$\Theta$ 
is the image of a congruence on some subalgebra of
the power
$\m a ^{ \lambda +1} $.

Hence, if $\mathcal V$ is a class closed under subalgebras and
products (in particular, if $\mathcal V$ is a variety),
then every weakly representable tolerance belonging to 
$\mathcal V$ is 
 the image of a congruence in $\mathcal V$.
\end{thm} 

\begin{proof}
Suppose that $\Theta = \bigcap_{i \in \lambda } \Theta _i $, where each
$\Theta_i $ has the form $  R_i \circ R_i\conv$, for certain reflexive compatible relations  
$R_i$.
Let  
$\m b$ be the subalgebra of $ \m a \times \m a^ \lambda \cong \m a ^{ \lambda +1} $ whose base set is
$B= 
\{ (a, (a_i) _{i \in  \lambda })   \mid 
a, a_i \in A, \text{ and }    a \Rx _i a_i, \text{ for each } i \in \lambda     \}$.  
The assumption that each $R_i$ is compatible implies that $\m b$ is indeed a subalgebra
of  $ \m a \times \m a^ \lambda  $.

Let $\varphi: \m b \to \m  a$ be the first projection, and
let $\beta$ be the kernel of the second projection 
$\pi : \m b \to \m a^ \lambda $. The same arguments as in the proof of
 Theorem \ref{thm} show that
$\varphi ( \beta ) = \Theta $.       
\end{proof}
 
The converse of Theorem \ref{propweakly} does not hold, in general; see 
Proposition \ref{ex}.  
However, we expect that the converse of  Theorem \ref{propweakly}
is true, under some mild assumptions on $\m a$ or $\mathcal V$.

Theorem \ref{propweakly} gives us the possibility of improving
Corollary \ref{thmim}. 

\begin{corollary} \labbel{thmimw} 
If all tolerances in a variety $\mathcal V$ are weakly representable, then 
the tolerances of $\mathcal V$ are the
 images of its congruences (indeed, it is enough to assume that $\mathcal V$ 
is closed under subalgebras and arbitrary products).
\end{corollary}

Extending Definition \ref{reprhom} in the natural way, 
if $\m a$ and $\m C$ are algebras, 
we say that \emph{a tolerance $\Theta$ on $\m a$ is the image of some tolerance $\Psi$ on $\m c$} if 
there is some surjective homomorphism $\psi: \m c \to \m a$
such that $\Theta = \{ ( \psi (a), \psi(b))\mid a \Psix b\} $.    
It is trivial to see that the image of some tolerance, in the above sense, is again a tolerance.

\begin{lemma} \labbel{3tol}
If 
$\Theta$, $\Psi$ and $\Phi$ are tolerances,
$\Theta$ is an image of  $\Psi$,
and $\Psi$ is an image of $\Phi$, then
$\Theta$  is an image of $\Phi$.
 \end{lemma}

\begin{proof} 
Let $\Theta$ be on $\m a$,
 $\Psi$ be on $\m c$,
and $\Phi$ be on $\m d$,
and let the assumption of the lemma be witnessed
by surjective homomorphisms
$\psi: \m c \to \m a$
and
$\varphi: \m d \to \m c$. Then
$ \psi \circ \varphi: \m d \to \m a$ witnesses that   $\Theta$ is an image of 
$\Phi$.
\end{proof} 

\begin{corollary} \labbel{imtol}
Let $\mathcal V$ be a class of algebras
closed under subalgebras and products (in particular,  a variety).
For every tolerance $\Theta$ in $ \mathcal V$, the following conditions are equivalent.
  \begin{enumerate}   
 \item 
 $\Theta$ is the image of a congruence in $\mathcal V$.
\item
$\Theta$ is the image of a representable tolerance in $\mathcal V$.
\item
$\Theta$ is the image of a weakly representable tolerance in $\mathcal V$.
  \end{enumerate} 

In particular, for every  $\mathcal V$ as above, the tolerances of $\mathcal V$ are the
 images of its congruences if and only if 
 the tolerances of $\mathcal V$ are the
 images of its (weakly) representable tolerances.
\end{corollary}

 \begin{proof} 
(1) $\Rightarrow $  (2) and (2) $\Rightarrow $  (3) are trivial, since
every congruence $\beta$ is representable (as $\beta= \beta \circ \beta $),
and since every representable tolerance is weakly representable.

(3) $\Rightarrow $  (1) If $\Theta$ is the image of  a weakly representable tolerance $\Psi$ in $\mathcal V$, then, by Theorem \ref{propweakly},  
$\Psi$ is the image of some congruence $\beta$ in $\mathcal V$,
hence, by Lemma \ref{3tol},  $\Theta$ is the image of $\beta$.
\end{proof}

\begin{cor} \labbel{equiv}
Suppose that 
 $\Theta$ is  a tolerance on the algebra $\m a$. Then the following conditions are equivalent.
  \begin{enumerate} 
\item
$\Theta$ 
is the image of a congruence on some subalgebra of
some power
$\m a ^I$, for some 
set $I$.
\item  
For every variety $\mathcal V$ such that $\m a\in \mathcal V$,
$\Theta$ 
is the image of a congruence in $\mathcal V$.
\item
 $\Theta$ 
is the image of a congruence 
in $\mathcal V( \m a)$,
the variety generated by $\m a$. 
 \end{enumerate} 
In all the preceding conditions we can equivalently replace the 
word ``congruence'' with either ``representable tolerance'' or 
``weakly representable tolerance''. 
 \end{cor}

\begin{proof}
(1)   $\Rightarrow $  (2) is obvious, since if $\m a \in \mathcal V$, then
every subalgebra of
$\m a ^I$ is in $\mathcal V$.

(2)   $\Rightarrow $  (3) is trivial.

(3)   $\Rightarrow $  (1) Let $\Theta$ 
be an image of $\gamma$, a congruence on 
$\m c \in \mathcal V( \m a)$.
By the HSP characterization of  $V( \m a)$,
there are a set $I$,  an algebra $\m b \subseteq \m a^I$, 
and a surjective homomorphism
$\varphi: \m b \to \m c$.
Then $\beta= \varphi ^{-1} ( \gamma ) 
= \{(b, b') \mid b,b'\in B \text{ and } (\varphi(b), \varphi(b')) \in \gamma  \}  $
is a congruence on $\m b$,
and $\varphi ( \beta )= \gamma $,
in the sense of Definition \ref{reprhom}.

Thus $\Theta$ 
is an image of $\gamma$,
which is an image of $\beta$,
hence, by Lemma \ref{3tol},  $\Theta$ is an image of 
$\beta$, a congruence on  $\m a^I$,
and (1) is proved.

The last statement is immediate
from Corollary \ref{imtol}. 
 \end{proof}

\section{Additional remarks} \labbel{add}

We first provide a generalization of Theorem \ref{lattices}.
Its proof  exploits exactly the only properties of lattices which were used in the proof of 
\ref{lattices}.

\begin{proposition} \labbel{lattgen}
Suppose that $\m a$ is an algebra with 
two binary operations $\vee$ and $\wedge$  (among possibly other operations),
and with a compatible binary relation $M$, which satisfy the following conditions:
  \begin{enumerate} 
   \item 
$a \vee a= a$, for every $a \in A$.
\item
$a \Mx (a \vee b)$, and 
$b \Mx (a \vee b)$,
for every $a,b \in A$.
\item 
$a= a \wedge c = c \wedge a $, for every $a,c \in A$ such that $a \Mx c$. 
  \end{enumerate}

Then every tolerance $\Theta$ of $\m a$ is representable,
and  is an image of some congruence on some subalgebra of $\m a \times \m a$.
\end{proposition} 

\begin{proof}
Same as the proof of Theorem 
\ref{lattices}, using $M$ in place of $\leq$:
$\Theta $ is representable as $  R \circ R\conv$, with $R= \Mx \cap~  \Theta $.
The last statement is immediate from
Theorem \ref{thm}.
\end{proof}

\begin{remark} \labbel{rmkcasin}
Condition (2) in Proposition \ref{lattgen} is satisfied 
in case 
$M$ is defined by  
\begin{equation*}
 a \Mx b \text{ if and only if  } a \vee b =b, 
 \end{equation*}   
and $\m a$ satisfies 
$ a \vee (a \vee b)= a \vee b$ and
 $ b \vee (a \vee b)= a \vee b$,
for every $a,b \in A$.
\end{remark}   

By Proposition \ref{lattgen}, and writing explicitly 
the condition that the $M$ given by  Remark \ref{rmkcasin}
is compatible, we get:

\begin{proposition} \labbel{lattgen2}
Suppose that $\m a$ is an algebra with 
(exactly) two binary operations $\vee$ and $\wedge$ satisfying the following conditions:
  \begin{enumerate} 
\item[(0)]
For every $a,a',b,b' \in A$, if 
$a \vee b= b$
and 
$a' \vee b'= b'$,
then 
$(a' \vee a') \vee (b \vee b')= b \vee b'$
and 
$(a' \vee a') \wedge (b \vee b')= b \wedge b'$.
   \item 
$a \vee a= a$, for every $a \in A$.
\item
$a \vee (a \vee b)= a \vee b$, and 
$b \vee (a \vee b)= a \vee b$,
for every $a,b \in A$.
\item 
For every $a,c \in A$, if 
$a \vee c= c$,
then $a= a \wedge c = c \wedge a $. 
  \end{enumerate}

Then every tolerance $\Theta$ of $\m a$ is representable,
and  is an image of some congruence on some subalgebra of $\m a \times \m a$.
\end{proposition}

We now show that the conditions exploited in the proof of Theorem \ref{thm} actually characterize
representable tolerances.

\begin{prop} \labbel{thm2} 
 Suppose that  $\Theta$ is  a  tolerance on the algebra $\m a$. Then  $\Theta$ is  representable if and only if 
$\Theta$ can be realized as  the image of a congruence on some subalgebra $\m b$ of
$\m a \times \m a$, such that $\m b$  contains $\Delta= \{(a,a) \mid a \in A\}$, and in such a way that  $\varphi$  and $\beta$ in
Definition \ref{reprhom} can be chosen to be, respectively, the first projection and the kernel of the second projection.
\end{prop} 

\begin{proof}
The construction used in the proof of Theorem \ref{thm} shows that if $\Theta$ is representable, then $\m b$, 
$\varphi$  and $\beta$ can be chosen to satisfy the desired requirements.

Conversely, suppose that we have $\m b \subseteq \m a \times \m a$, 
$\varphi$  and $\beta$ satisfying the conditions in the statement of the 
proposition.
 Being a subalgebra of $ \m a \times \m a$,
 $\m b$ can be thought of as a compatible relation on $\m a$. We shall take $R=B$. 
Since $B$ contains $\Delta$, then $R$ is reflexive.
By assumption, $a \Thetax c$ if and only if  $(a, c) \in \varphi( \beta )$. Noticing that the equivalence of items (5) and (2) in the proof of Theorem \ref{thm} holds also in the present situation, we get that
$a \Thetax c$ if and only if
there is $b\in A$ such that 
$(a , b) \in R$ and $ (b, c) \in R\conv$.
This means exactly that $\Theta= R \circ R\conv$. 
\end{proof}

We now show that the converse of Theorem \ref{propweakly}
fails in a large class of algebras. 

\begin{proposition} \labbel{ex}
For every set $A$, and every reflexive and symmetric relation 
$\Theta$ on $A$ which is not transitive, there is an algebra 
$\m a$ with base set $A$ and such that $\Theta$ is a tolerance 
on $\m a$ which is not weakly representable, but $\Theta$ is the image of some congruence
on some subalgebra of some power of $\m a$.
 \end{proposition}

  \begin{proof} 
For every $a, b \in A$ such that $a \Thetax b$,
and for every   function $f: A \to \{ a, b \} $,
 add to $A$ a unary function symbols representing $f$.
It is easy to see that $\Theta$ is a tolerance on the algebra thus obtained, and
that $\Theta$ is not weakly representable. Indeed, every nontrivial compatible relation $R$ on $\m a$ 
contains $\Theta$, and, since $\Theta$ is not transitive, then
$\Theta \subset \Theta \circ \Theta \subseteq R \circ R\conv$  (see \cite[Proposition 12]{acta} for more details).

Consider $ \mathcal V(\m a)$, the variety generated by $\m a$. 
Since    $ \mathcal V(\m a)$ is unary, then, by 
\cite[Corollary 4.4]{CK}, tolerances are images of congruences 
in $ \mathcal V(\m a)$. Then Corollary \ref{equiv}(3) $\Rightarrow $  (1)
implies that  
$\Theta$ is the image of some congruence
on some subalgebra of some power of $\m a$.
\end{proof}  

On the other hand, under certain conditions,
the converse of Theorem \ref{propweakly} does hold.

\begin{proposition} \labbel{3perm}
Suppose that $\m a$ is an algebra in a 3\hspace{1.5 pt}-permutable variety $\mathcal V$,
and $\Theta$ is a tolerance on $\m a$. Then the following conditions are equivalent.
\begin{enumerate}
   \item  
$\Theta$ is representable.
\item
$\Theta$ is weakly representable.
\item
$\Theta$ is the image of some congruence on
some subalgebra of $\m a \times \m a$. 
\item
$\Theta$ is the image of some congruence in
$\mathcal V( \m a)$.
\item
$\Theta$ is a congruence of $\m a$. 
 \end{enumerate} 

If we only assume that
every subalgebra of $\m a \times \m a$ has 3\hspace{1.5 pt}-permutable congruences,
then Conditions (1), (3) and (5) above are still equivalent.
 \end{proposition}

  \begin{proof}
(1) $\Rightarrow $  (2) and (3) $\Rightarrow $  (4) are trivial.

(1) $\Rightarrow $  (3) and (2) $\Rightarrow $  (4) follow from, respectively,
Theorems \ref{thm} and \ref{propweakly}.  

(4) $\Rightarrow $  (5) Clearly $\mathcal V(\m a) $, being a subvariety of $  \mathcal V$, 
is 3\hspace{1.5 pt}-permutable, too, hence $\Theta$ is the image of some congruence on 
some algebra with 3\hspace{1.5 pt}-permuting congruences. But it is well-known that
this implies that $\Theta$ is a congruence, see Theorem 1.10 in \cite[Chapter 7]{J}. 

(5) $\Rightarrow $  (1) is trivial, since if $\Theta$ is a congruence, then
$\Theta = \Theta \circ \Theta\conv$. 

Under the assumption that every subalgebra of $\m a \times \m a$ has 3\hspace{1.5 pt}-permutable congruences, (3) $\Rightarrow $  (5) holds, again by Theorem 1.10 in \cite[Chapter 7]{J}. The implications (1) $\Rightarrow $  (3) and (5) $\Rightarrow $  (1) do not use 3\hspace{1.5 pt}-permutability at all.
 \end{proof}

We expect that parts of Proposition \ref{3perm}  hold under assumptions
weaker than 3\hspace{1.5 pt}-permutability. However, globally
(that is, if we ask that the conditions hold for every tolerance in a 
3\hspace{1.5 pt}-permutable variety),
Proposition \ref{3perm} is essentially an empty result, in the sense
that the conditions hold only in permutable varieties (in which they are trivially true).

\begin{corollary} \labbel{npermv} 
Suppose that $\mathcal V$ is an $n$\hspace{1.5 pt}-permutable variety,
for some $n$.  Then the following conditions are equivalent. 
\begin{enumerate} 
\item
Every tolerance in $\mathcal V$ is weakly representable.
\item
Every tolerance in $\mathcal V$ is representable.
\item
Every tolerance in $\mathcal V$ is the image of a congruence in $\mathcal V$.
  \item 
$\mathcal V$ is permutable.
\item
Every tolerance in $\mathcal V$ is a congruence.
  \end{enumerate}
\end{corollary} 

\begin{proof}
(1) $\Rightarrow $  (2) and (5) $\Rightarrow $  (1) are trivial.

(2) $\Rightarrow $  (3) follows from Theorem \ref{thmim}.

(3) $\Rightarrow $  (4) is \cite[Theorem 5.3]{CK}.

 (4) $\Rightarrow $  (5) is immediate from a classical result from 
\cite{We}, parts of which are due independently to G. Hutchinson \cite{Hu}.
Actually, Conditions (4) and (5) are equivalent for every variety, as follows
easily from the above papers, and explicitly stated, e.~g.,  in \cite{Ch}. 
 \end{proof}

\begin{problems} \labbel{more}
Notice that, again by \cite{CZ,Clattices}, tolerances in lattices satisfy
a property stronger  than representability.
Indeed, if $\Theta$ is a tolerance on a lattice $\m l$, then

\begin{equation} \labbel{eq}
\text{there is a compatible relation $R$  such that } \Theta = (R \circ R\conv) \cap (R\conv \circ R),
 \end{equation}       

or even 
\begin{multline} \labbel{eqm}
\text{there is a compatible relation $R$  such that   $a \Thetax  b $ if and only if } \\ 
\text{there are $c$ and $d$ such that 
$a \Rx c \Rx\conv b$, $a \Rx\conv d \Rx b$, and $d \Rx c$}. 
 \end{multline}    

(just take 
$R = \Thetax \cap \leq$, 
$c= a \vee b$ and $d= a \wedge b$) 

Which parts of the theory of tolerances on lattices follow
just from the assumption (\ref{eq})  or (\ref{eqm})?

Notice that we do not need all the axioms for lattices, in order to get
 (\ref{eq}) above: the properties listed in Proposition \ref{lattgen}, together with their duals
suffice. 
 \end{problems}


\begin{thebibliography}{CCHL}

\bibitem[C]{Ch} 
I. Chajda,
\emph{Recent results and trends in tolerances on algebras and varieties}, in
   Finite algebra and multiple-valued logic (Szeged, 1979),   69--95,
   Colloq. Math. Soc. J\'anos Bolyai, 28, North-Holland, Amsterdam-New
   York, 1981.

\bibitem[CCH]{CCH} I. Chajda, G. Cz{\'e}dli, R. Hala{\v{s}}, 
\emph{Independent joins of tolerance factorable varieties}, submitted.

\bibitem[CCHL]{CCHL} I. Chajda, G. Cz{\'e}dli, R. Hala{\v{s}}, P. Lipparini,  \emph{Tolerances as images of congruences in varieties defined
by linear identities}, submitted.

\bibitem[CZ]{CZ}  I. Chajda, B. Zelinka,
   \emph{Tolerance relation on lattices}, 
 Casopis Pest. Mat. \textbf{99} (1974), 394-399.

\bibitem[Cz]{Clattices}  G. Cz{\'e}dli, \emph{Factor lattices by tolerances},
  Acta Sci. Math. (Szeged) \textbf{44} (1982), 35--42.
\bibitem[CG]{CG} G. Cz{\'e}dli, G. Gr{\"a}tzer,   \emph{Lattice tolerances and congruences}, Algebra Universalis \textbf{66} (2011), 5--6.

\bibitem[CHL] {CHL} G. Cz\'edli, E. Horv\'ath, P. Lipparini, \emph{Optimal Mal'tsev conditions
for congruence modular varieties},
Algebra Universalis \textbf{53} (2005),  267--279.

\bibitem[CK]{CK} G. Cz{\'e}dli, G., E. W. Kiss, \emph{Varieties whose tolerances are homomorphic
images of their congruences}, preprint, arXiv:1204.2228.  


\bibitem[G]{Gu}
H.~P. Gumm, \emph{Geometrical methods in congruence modular algebras},
  Memoirs of the American Mathematical Society \textbf{45} (1983), no.~286,
  viii+79. 


\bibitem[HM]{HM} 
J. Hagemann, A. Mitschke,
   \emph{On $n$-permutable congruences},
   Algebra Universalis \textbf{3} (1973), 8--12.

\bibitem[HMK]{HMK}
D. Hobby and R. McKenzie, \emph{The structure of finite algebras},
  Contemporary Mathematics, vol.~76, American Mathematical Society, Providence,
  RI, 1988. 

\bibitem[H]{Hu} G. 
Hutchinson,  \emph{Relation categories and coproduct congruence categories in universal algebra}, Algebra Universalis \textbf{32}, 609-647 (1994). 

\bibitem[J]{J} J. Je{\v{z}}ek,   \emph{Universal Algebra} (2008), book available at http:// www.karlin.mff.cuni.cz/\textasciitilde jezek/.  


\bibitem[KK]{KK} K. A. Kearnes and E. W. Kiss,
{\em The Shape of Congruence Lattices}, 
in print.

\bibitem[L]{acta} P.  Lipparini,
\emph{From congruence identities to tolerance identities}, Acta Sci. Math. (Szeged) \textbf{73} (2007), 31--51.

\bibitem[S]{Sm} J. D.~H. Smith, \emph{Mal'cev varieties}, Springer-Verlag, Berlin,
  1976, Lecture Notes in Mathematics, Vol. 554.

\bibitem[T]{Ts}
S.~T. Tschantz, \emph{More conditions equivalent to congruence modularity},
  Universal algebra and lattice theory (Charleston, S.C., 1984), Lecture Notes
  in Math., Vol. 1149, Springer, Berlin (1985), 270--282. 

 \bibitem[W]{We} H. Werner, 
\emph{A Mal'cev condition for admissible relations},
Algebra Universalis \textbf{3} (1973), 263.



\end{thebibliography}
\end{document}